\documentclass[11pt]{amsart}
\usepackage{pstricks,pst-plot}

\definecolor{Black}{cmyk}{0,0,0,1}
\definecolor{OrangeRed}{cmyk}{0,0.6,1,0}            
\definecolor{DarkBlue}{cmyk}{1,1,0,0.20}
\definecolor{myblue}{rgb}{0.66,0.78,1.00}
\definecolor{Violet}{cmyk}{0.79,0.88,0,0}
\definecolor{Lavender}{cmyk}{0,0.48,0,0}

\parskip=\smallskipamount

\newtheorem{theorem}{Theorem}[section]
\newtheorem{lemma}[theorem]{Lemma}

\theoremstyle{definition}

\newtheorem{problem}[theorem]{Problem}

\author{J. E. Forn\ae ss}
\thanks{This article was written as part of the international research program "Several Complex Variables and Complex Dynamics" at the Centre for Advanced Study at the Norwegian Academy of Science and Letters in Oslo during the academic year 2016/2017.}
\thanks{Both authors are supported by the NRC grant number 240569}

\address{J. E. Forn\ae ss: Department of Mathematical Sciences, Norwegian University of Science
and Technology, 7491 Trondheim, Norway. }
\author{E. F. Wold}
\address{E. F. Wold: Department of Mathematics, University of Oslo, PO-BOX 1053 Blindern, 0316 Oslo, Norway.}

\newcommand{\bea}{\begin{eqnarray*}}
\newcommand{\eea}{\end{eqnarray*}}
\numberwithin{equation}{section}

\title[The Squeezing Function]{A non-strictly pseudoconvex domain for which the squeezing function tends to one towards the boundary}

\begin{document}

\maketitle

\begin{abstract}
In recent work by Zimmer it was proved that if $\Omega\subset\mathbb C^n$ is a bounded 
convex domain with $C^\infty$-smooth boundary, then $\Omega$ is strictly pseudoconvex 
provided that the squeezing function approaches one as one approaches the boundary.   We show that 
this result fails if $\Omega$ is only assumed to be $C^2$-smooth.
\end{abstract}

\section{Introduction}

We recall the definition of the squeezing function $S_\Omega(z)$ on a bounded domain $\Omega\subset\mathbb C^n$.
If $z\in\Omega$, and $f_z:\Omega\rightarrow\mathbb B^n$ is an embedding with $f_z(z)=0$, we set 
\begin{equation}
S_{\Omega,f_z}(z):=\sup\{r>0:B_r(0)\subset f_z(\Omega)\}, 
\end{equation}
and then
\begin{equation}
S_\Omega(z):=\sup_{f_z}\{S_{\Omega,f_z}(z)\}.
\end{equation}
A guiding question is the following: which complex analytic properties of $\Omega$ are encoded by the behaviour
of $S_\Omega$?  For instance, if $S_\Omega$ is bounded away from zero, then $\Omega$ is necessarily 
pseudoconvex,  and the Kobayashi-, Carath\'{e}odory-, Bergman- and the K\"{a}hler-Einstein metric are complete, and they are pairwise quasi-isometric (see \cite{Yeung}).   Recently, Zimmer \cite{Zimmer} proved that if 
\begin{equation}
\lim_{z\rightarrow b\Omega}S_\Omega(z)=1
\end{equation}
for a $C^\infty$-smooth, bounded convex domain, then $\Omega$ is necessarily strictly pseudoconvex.  In this short 
note we will show that this does not hold for $C^2$-smooth domains.

\begin{theorem}\label{main}
There exists a bounded convex $C^2$-smooth domain $\Omega\subset\mathbb C^n$ which is not strongly pseudoconvex, but 
\begin{equation}
\lim_{z\rightarrow b\Omega}S_\Omega(z)=1,
\end{equation}
where $S_\Omega(z)$ denotes the squeezing function on $\Omega$.
\end{theorem}

For further results about the squeezing function the reader may also consult the references \cite{DFW}, \cite{DGZ1},\cite{DGZ2},\cite{FR},\cite{FornæssWold1},\cite{KZ},\cite{LSY}, \cite{Yeung}, \cite{Zimmer}.   In the last section we will post some open problems.

\section{The construction}

\subsection{The construction in $\mathbb R^n$ and curvature estimates}

We start by describing a construction of a convex domain $\Omega$ in $\mathbb R^n$ with a single non-strictly 
convex point.  Afterwards we will explain how to make the construction give the conclusion of Theorem 
\ref{main} for each $n=2m$, when we make the identification with $\mathbb C^m$.  \

Let $x={x_1,...,x_n}$ denote the coordinates on $\mathbb R^n$.  For any $k\in\mathbb N$ we 
let $B_k$ denote the ball
\begin{equation}
B_k:=\{x\in\mathbb R^n:x_1^2+\cdot\cdot\cdot + x_{n-1}^2 + (x_n-k)^2<k^2\}.
\end{equation}
On some fixed neighbourhood of the origin, each boundary $bB_k$ may be written as a graph 
of a function 
\begin{equation}
x_n=\psi_k(x')=\psi_k(x_1,...,x_{n-1})= k-\sqrt{k^2-|x'|^2} =\frac{1}{2k}|x'|^2 + \mathrm{h.o.t.}
\end{equation}
Fix a smooth cut-off function $\chi(x')=\chi(|x'|)$ with compact support in $\{|x'|<1\}$
which is one near the origin.  We will create a new limit graphing function $f(x')$ by 
subsequently gluing the functions $\psi_k$ and $\psi_{k+1}$ by setting 
\begin{equation}
g_k(x') = \psi_{k}(x') + \chi(\frac{x'}{\epsilon_k})(\psi_{k+1}(x')-\psi_k(x')),
\end{equation}
where the sequence $\epsilon_k$ will converge rapidly to zero, and the boundary of our domain $\Omega$ will be defined (locally) as the graph $\Sigma$ of the function $f$ defined as follows: start by setting $f_k:=\psi_k$ for some $k\in\mathbb N$.  Then 
define $f_{k+1}$ inductively by setting $f_{k+1}=f_k$ for $\|x'\|\geq\epsilon_k$ and then $f_{k+1}=g_k$
for $\|x'\|<\epsilon_k$.  Finally we set $f=\lim_{k\rightarrow\infty} f_k$. \

To ensure that $\Omega$ is convex we will need to estimate the curvature of $\Sigma$, and estimates of 
the curvature of the partial graphs $\Sigma_k=\{x,g_k(x)\}$ will be necessary to prove Theorem \ref{main}.
Informally our goal is to show the following: \emph{There exist $N,m\in\mathbb N, N>m$, such that if $k\geq N$ and 
if $\epsilon_k$ is sufficiently small (depending on $k$), then 
$\Sigma_k$ curves, at every point and in all directions, more than $bB_{k+m}$ and less than $bB_{k-m}$}. \

We make this more precise.  The surface $\Sigma_k$ has a defining function $\rho_k(x)=g_k(x')-x_n$.
If $v_{p}$ is a tangent vector to $\Sigma_k$ at $p=(x',g_k(x))$, the curvature of $\Sigma_k$ in 
the direction of $v_{p}$ is defined as 
\begin{equation}\label{curvature1}
\kappa^{\Sigma_k}_{p}(v_p):=\frac{H\rho_k(p)(v_p)}{\|\nabla\rho_k(p)\|\|v_p\|^2},
\end{equation}
where $\nabla\rho_k$ is the gradient, and $H\rho_k$ is the Hessian of $\rho_k$ (which is equal to the Hessian of $g_k$).
The curvature \eqref{curvature1} depends only on the direction of $v_p$, and the curvature 
of $bB_k$ is $\frac{1}{k}$ at all points and in all directions.  The precise statement 
of our goal stated above is \

\begin{lemma}\label{curvature}
Let $\psi_k$ and $\chi$ be defined as above for $k\in\mathbb N$.  There exist $N,m\in\mathbb N, N>m$, such that 
if each $\epsilon_k$ is sufficiently small (depending on $k$), and $k\geq N$, then 
\begin{equation}
\frac{1}{k+m}\leq\kappa^{\Sigma_k}_p(v_p)\leq\frac{1}{k-m},
\end{equation}
for all $v_p$ tangent to $\Sigma_k$.
\end{lemma}

If is now easy to see that if $\epsilon_k\searrow 0$ sufficiently fast, then $\Omega$ is convex, and 
strictly convex away from the origin.    If we let $\Omega_k$ denote the domain whose 
boundary near the origin is given by the graph of $f_k$, we see that $\Omega_k$ is 
strictly convex, the Hessian being positive definite everywhere.   Morover $\Omega=\cup_k\Omega_k$, and so $\Omega$ is convex.

\begin{proof}(of Lemma \ref{curvature})
When we estimate the curvature we may assume that the functions $g_k$ are simply
\begin{equation}
g_k(x')=\psi_{k}(x')-\chi(\frac{x'}{\epsilon_k})(\frac{1}{2k(k+1)})|x'|^2 = : \psi_k(x') + \sigma_k(x'),
\end{equation}
since the higher order terms missing in this expression of $g_k$ can be made 
insignificant by choosing $\epsilon_k$ small enough.  
Because of the $|x'|^2$-term it is easy to see that 
\begin{equation}
dg_k(x')=d\psi_k(x') + \triangle_k(x'),
\end{equation}
and 
\begin{equation}
Hg_k(x')=H\psi_k(x') + h_k(x'),
\end{equation}
where the coefficients in both $\triangle_k$ and $h_k$ are of order of magnitude $\frac{1}{k^2}$
independently of $k$ and of the choice of a small $\epsilon_k$.  \

Fix a point $x'$ and a vector $v\in\mathbb R^{n-1}$ with $\|v\|=1$.   Then a tangent vector 
$v_p$ at the point $(x',g_k(x'))$ is given by 
\begin{equation}
v_p=(v,dg_k(x')(v))=(v,d\psi_k(x')(v) + \triangle_k(x')(v)).
\end{equation}
Estimating the curvature we see that 
\begin{align*}
\kappa^{\Sigma_k}_{p}(v_p) & = \frac{(H\psi_k(x') + h_k(x'))(v_p)}{\|\nabla\rho_k(p)\|\|v_p\|^2}\\
& = \frac{(H\psi_k(x'))((v,d\psi_k(x')v) + (0',\triangle_k(x')(v)))}{\|{-\bf e_n} + \nabla\psi_k(p) + \nabla\sigma_k(x')\|\|(v,d\psi_k(x')(v))+(0',\triangle_k(x'))\|^2}\\
& + O(\frac{1}{k^2})\\
& = \frac{(H\psi_k(x'))((v,d\psi_k(x')v))}{\|-{\bf e_n} + \nabla\psi_k(x')\|(1 + O(\frac{1}{k^2}))\|(v,d\psi_k(x')(v))\|^2(1+O(\frac{1}{k^2}))^2}\\
& +O(\frac{1}{k^2})\\
& =  \frac{(H\psi_k(x'))((v,d\psi_k(x')v))}{\|-{\bf e_n} + \nabla\psi_k(x')\|\|(v,d\psi_k(x')(v))\|^2} + O(\frac{1}{k^2})\\
& = \frac{1}{k} + O(\frac{1}{k^2}),
\end{align*}
where the term $\frac{1}{k}$ comes from the fact that the expression above is the formula for the curvature of a ball 
of radius $k$.   From this it is straightforward to deduce the existence of an $m$ such that the lemma holds. 
\end{proof}

\subsection{The squeezing function on $\Omega$}

We will now explain why the squeezing function goes to one uniformly as we approach 
$b\Omega$ provided that the $\epsilon_k$'s decrease sufficiently fast.   Let $N,m$ be as 
in Lemma \ref{curvature}, and start by setting $f_k=\psi_k$ for some $k>N$.   \

Fix some small $\delta_k>0$.  By Lemma \ref{curvature}, if $\epsilon_k$ is small enough, 
we can for each $p=(x',x_n)\in b\Omega_k, \|x'\|<\delta_k$, find a ball $B$ of radius $k+m$ containing 
$\Omega_k$ such that $p\in bB$.  By the same lemma we can for each such $p$ also 
find a local piece of a ball of radius $k-m$ touching $p$ from the inside of $\Omega_k$, and 
the size of the local ball is uniform.    So using Lemma \ref{mainestimate} we may find 
a $t_k>0$ small enough such that 
\begin{equation}\label{est1}
S_{\Omega_k}(x',x_n)\geq 1 - \frac{m}{(k+m)}
\end{equation}
if $x_n\leq t_k$.  \

Next, again by Lemma \ref{curvature}, we find a $\delta_{k+1}<\delta_k$ such that 
if $\epsilon_{k+1}$ is small enough, then for each $p=(x',x_n)\in b\Omega_{k+1}$ 
with $\|x'\|<\delta_{k+1}$, we may oscillate with balls of radius $k+1-m$ and 
$k+1+m$ respectively.  So there is a $t_{k+1}<t_k$ such that 
\begin{equation}
S_{\Omega_{k+1}}(x',x_n)\geq 1 - \frac{m}{(k+1+m)}
\end{equation}
if $x_n\leq t_{k+1}$.  Furthermore, by further decreasing $\epsilon_{k+1}$ we 
can keep the estimate \eqref{est1} with $\Omega_k$ replaced by $\Omega_{k+1}$.  The reason is the following.  First of all, 
by \cite{FornæssWold1} there exists a constant $C_k$ such that 
\begin{equation}
S_{\Omega_k}(z)\geq 1 - C_k\cdot \mathrm{dist}(z,b\Omega_k),
\end{equation} 
and near any compact $K\subset b\Omega_k$ away from $0$, this estimate is not going to be 
disturbed by a small perturbation of $b\Omega_k$ near the point $0$; the estimate 
is obtained by using oscillating balls at points of $K$ whose boundaries will stay bounded 
away from $0$.  Furthermore, on any compact subset of $\Omega_k$ we have 
that $S_{\Omega_{k+1}}\rightarrow S_{\Omega_k}$ as $\epsilon_{k+1}\rightarrow 0$.  \

Continuing in this fashion, we obtain a decreasing sequence $0<t_j<t_{j+1}, j=k,k+1,...$, and 
an increasing sequence of domains $\Omega_j$, such that for each $j$ we have that 
\begin{equation}
S_{\Omega_j}(x',x_n)\geq 1 - \frac{m}{(k+i+m)}
\end{equation}
for $t_{k+i}\leq x_n\leq t_{k+i-1}$, for $i\leq j$.  The result now follows from Lemma \ref{increase}.  \

\section{Lemmata}

Let $0<s<1/2, 0<d<r<1$, and set $B_s=B(s,1-s)$.  Furthermore we set 
\begin{equation}
B_{s,d}=B_s\cap\{(z_1,z')\in\mathbb B^n:\mathcal{R}(z_1)>d\}.
\end{equation}

\begin{lemma}\label{mainestimate}
If $B_{s,d}\subset\Omega\subset\mathbb B^n$, and if $r>1-\frac{sd}{4}$, then
$S_\Omega(r,0)>1-s$. 
\end{lemma}
\begin{proof}
Set $\mu=1-s$ and $\eta=\frac{d}{2}$, and then 
\begin{equation}
B^{\mu}_\eta=\{(z_1,z')\in\mathbb C^n:|z_1-(1-\eta)|^2 + \frac{\eta}{\mu}|z'|^2<\eta^2\}.
\end{equation}
Then certainly $\mathcal R(z_1)>d$ on $B^{\mu}_\eta$, and we also have that $B^{\mu}_\eta\subset B_s$.
To see the latter, we translate the two balls sending $(1,0')$ to the origin, where they are defined by 
\begin{equation}
\tilde B_s=\{(z_1,z'):2\mu\mathcal R(z_1) + |z|^2<0\},
\end{equation}
and 
\begin{equation}
\tilde B^{\mu}_\eta=\{(z_1,z'):2\eta\mathcal Re(z_1)+|z_1|^2 + \frac{\eta}{\mu}|z'|^2<0\}.
\end{equation}
And 
\begin{align*}
2\eta\mathcal Re(z_1)+|z_1|^2 + \frac{\eta}{\mu}|z'|^2<0 & \Rightarrow 2\eta\mathcal Re(z_1)+\frac{\eta}{\mu}|z_1|^2 + \frac{\eta}{\mu}|z'|^2<0\\
& \Leftrightarrow 2\mu\mathcal R(z_1) + |z|^2<0.
\end{align*}
According to Lemma 3.5 in \cite{FornæssWold1} we have that 
\begin{equation}
S_\Omega(r,0)\geq \sqrt{\mu}\sqrt{1-2(1-r)\frac{1}{\eta}}=\sqrt{(1-s)(1-\frac{4(1-r)}{d})},
\end{equation}
from which the lemma follows easily.
\end{proof}

\begin{lemma}\label{increase}
Let $\Omega_j\subset\Omega_{j+1}$ for $j\in\mathbb N$,  set $\Omega=\cup_j\Omega_j$, 
and assume that $\Omega$ is bounded. 
Let $z\in\Omega$, and assume that $S_{\Omega_j}(z)> 1 - \delta$ for all $j$ large 
enough so that $z\in\Omega_j$.  Then $S_\Omega(z)\geq 1 - \delta$.
\end{lemma}
\begin{proof}
Let $f_j:\Omega_j\rightarrow\mathbb B^n$ be an embedding such that $f_j(z)=0$
and $B_{1-\delta}(0)\subset f_j(\Omega_j)$.  By passing to a subsequence we 
may assume that $f_j\rightarrow f:\Omega\rightarrow\mathbb B^n$ u.o.c., with 
$f(z)=0$.  Setting $g_j=f_j^{-1}:B_{1-\delta}(0)\rightarrow\Omega$ we may 
also assume that $g_j\rightarrow g$ u.o.c.  Then $f|_{g(B_{1-\delta}(0))}=g^{-1}$, from 
which the result follows.  
\end{proof}

\section{Some open problems}

\begin{problem}
Does Zimmer's result hold for pseudoconvex domains of class $C^\infty$?
\end{problem}
\begin{problem}
How much smoothness is needed for Zimmer's result hold for convex/pseudoconvex domains?
\end{problem}
\begin{problem}
Let $\Omega\subset\mathbb C^2$ be a bounded pseudoconvex domain of class $C^\infty$. 
Is $S_\Omega(z)$ bounded away from zero? 
\end{problem}
Yeung \cite{Yeung} showed that the answer is yes for strongly convex domains in $\mathbb C^n$, 
and Kim-Zhang \cite{KZ} and Deng-Guan-Zhang \cite{DGZ2} showed that the answer is yes
for strictly pseudoconvex domains.   On the other hand, Forn\ae ss-Rong \cite{FR} showed 
that the answer is no for $n\geq 3$.  \

Quantifying the asymptotic behaviour of the squeezing function, Forn\ae ss-Wold \cite{FornæssWold1}
showed that 
\begin{itemize}
\item[(i)] $S_\Omega(z)\geq 1-C\mathrm{dist}(z,b\Omega)$, and 
\item[(ii)] $S_\Omega(z)\geq 1-C\sqrt{\mathrm{dist}(z,b\Omega)}$,
\end{itemize}
for strongly pseudoconvex domains 
of class $C^4$ and $C^3$ respectively.   Diederich-Forn\ae ss-Wold \cite{DFW} showed that 
if the the squeezing function approaches one essentially faster than (i), then $\Omega$
is biholomorphic to the unit ball.  
\begin{problem}
What is the optimal estimate for the squeezing function for strictly pseudoconvex domains 
of class $C^k$ with $k<4$?
\end{problem}
Let $\phi:\mathbb B^2\rightarrow\mathbb C^2$ be defined as $\phi(z_1,z_2):=({z_1,-z_2\log(z_1-1)})$.
Then $\Omega:=\phi(\mathbb B^2)$ is of class $C^1$, and $(1,0)$ is a non-strictly 
pseudoconvex boundary point of $\Omega$.  So $S_\Omega$ being identically equal to one does 
not even imply strict pseudoconvexity in the case of $C^1$-smooth boundaries.  

\begin{problem}
Let $\phi:\mathbb B^n\rightarrow\Omega$ be a biholomorphism, and assume that $\Omega$
is a bounded $C^2$-smooth domain.  Is $\Omega$ strictly pseudoconvex?
\end{problem}

\end{document}